\documentclass{article}
\usepackage[utf8]{inputenc}
\usepackage[T1]{fontenc}
\usepackage{lmodern} 
\usepackage{microtype} 
\usepackage[margin=1in]{geometry} 
\usepackage{setspace}
\usepackage{float}
\usepackage{xcolor} 
\usepackage{parskip} 

\usepackage{amsmath}   
\usepackage{amssymb}   
\usepackage{amsthm}    
\usepackage{mathtools} 
\usepackage{mathrsfs}  
\usepackage{tikz-cd}   
\usepackage{bm}        

\usepackage{titlesec} 
\titleformat{\section}{\Large\bfseries\color{blue!70!black}}{\thesection}{1em}{}
\titleformat{\subsection}{\large\bfseries\color{blue!60!black}}{\thesubsection}{1em}{}

\usepackage{enumitem} 
\setlist[itemize]{itemsep=0.2em, topsep=0.2em, leftmargin=1.5em}

\newcommand{\defeq}{\mathrel{\mathop:}=}

\usepackage{hyperref}
\hypersetup{
    colorlinks=true,
    linkcolor=blue!70!black, 
    filecolor=magenta,      
    urlcolor=cyan,
    pdftitle={Mathematical Notes},
    bookmarksnumbered=true, 
}
\usepackage[
  backend=biber,
  style=alphabetic,
  maxbibnames=999
]{biblatex}
\theoremstyle{plain}
\newtheorem{theorem}{Theorem}[section]
\newtheorem{lemma}[theorem]{Lemma}

\newtheorem{proposition}[theorem]{Proposition}

\newtheorem{definition}[theorem]{Definition}

\newtheorem{remark}[theorem]{Remark}
\newtheorem{conj}[theorem]{Conjecture}

\numberwithin{equation}{section}

\newcommand{\E}{\mathbb{E}}
\newcommand{\bbP}{\mathbb{P}}

\title{Generating and generalizing MSTD sets through Markov processes}
\author{Frank He (fsh30@cam.ac.uk), Karol Daniewski (kkd2@williams.edu), \\Steven J. Miller (sjm1@williams.edu)}
\date{July 2026}

\begin{document}

\maketitle

\section*{Abstract}
The classical More Sums Than Differences (MSTD) problem studies finite
sets \(A\subset \{0, 1, \dots, n\}\) for which \begin{equation*}\label{eq:abstract-mstd}
|A+A| \ >\ |A-A|,
\end{equation*} where
$A+A\ \defeq\ \{a_1+a_2:a_1,a_2\in A\}$ and $A-A\ \defeq\ \{a_1-a_2:a_1,a_2\in A\}$. As
addition is commutative and subtraction is not, it was conjectured that as
$n \to \infty$ almost all subsets $A$ chosen uniformly from the power set of $\{0, 1,
\dots, n\}$ (i.e., each $k$ is in $A$ with probability 1/2) are difference
dominated, and it was thus a surprise when Martin and O'Bryant proved a
positive percentage of sets are sum-dominant.

We greatly generalize this model by introducing a Markov-chain framework, where the
classical MSTD model is now just a special case. Let \((X_i)_{i=0}^n\) be
a stationary two-state Markov chain with transition matrix
\begin{equation*}\label{AbstractTransition}
P \ \defeq\ \begin{pmatrix}p&1-p\\1-q&q\end{pmatrix},
\quad\text{where }p,q\in(0,1).
\end{equation*}
We include $i$ in $A$ exactly when $X_i=1$, and define
\begin{equation*}\label{AbstractSet}
A \ \defeq\ \{i\in\{0,\ldots,n\}:X_i=1\}.
\end{equation*} The usual
independent Bernoulli model is recovered when consecutive inclusion
decisions are independent, equivalently when \(p=1-q\). In particular, the
uniformly random subset model corresponds to the case \(p=q=1/2\). Using the
fringe-middle method from the MSTD literature, we show that the middle
sums and differences are filled with high probability, so the comparison
between \(\left|A+A\right|\) and \(\left|A-A\right|\) is again governed by endpoint fringes. By fringe manipulation, we prove that the probabilities of sum-dominant, difference-dominant, and balanced sets tend to strictly positive limits
as $n\to\infty$. We also give numerical estimates of these three probabilities
for finite $n$ over a range of values of $p$ and $q$. Through combinatorial methods, we find a closed-form expression for $\mathbb{E}[|A_n-A_n|-|A_n+A_n|]$ as \(n\to\infty\).
\newpage
\section{Introduction}
\begin{definition}
    Let $A\subseteq S\defeq \{0,\ldots,n\}$. Define $A+A\defeq\{a+a'\mid a,a'\in A\}$, $A-A\defeq\{a-a'\mid a,a'\in A\}$. We call $A$
    \begin{enumerate}
    \item Sum-dominant (or More Sum Than Differences, MSTD) if $|A+A|> |A-A|$,
    \item Difference-dominant if $|A+A| < |A-A|$,
    \item Balanced if $|A+A| = |A-A|$.
\end{enumerate}
    
\end{definition}
\begin{remark}
    Define $[a,b]\defeq\{a,a+1,a+2,\ldots,b\}$ for integers $a\leq b$.
\end{remark}

Addition is commutative, whereas subtraction is not. So, with the same number of terms, one would expect more distinct differences to be produced. Nevertheless, Marica \cite{marica1969} proved the existence of MSTD sets, giving the following example:
\begin{equation}\{1,2,3,5,8,9,13,15,16\}.\end{equation}

While differences typically dominate, Martin and O'Bryant \cite{martin2006mstd} proved that a uniformly chosen $A\in \mathcal{P}(S)$ is expected to have $4$ more differences than sums as $n\to\infty$. However, in the same paper, they also proved a positive proportion are sum-dominant sets.

Zhao \cite{Zhao2011} extended this using results from Hegarty \cite{Hegarty2007} by proving that the proportion of sum-dominant sets has a positive limit. He also numerically showed that the limiting value is about $0.04\%$.

The principal explanation for these results is the distinction between the \emph{fringe} and the \emph{middle}. There are many ways to form a sum or difference closer to the middle of a range, giving a higher probability of realizing it. In contrast, extreme sums and differences depend on the inclusion in $A$ of elements close to $0$ and $n$. Hence, the comparison between $|A+A|$ and $|A-A|$ is determined by endpoint data.

This principle also underlies explicit constructions. For example, Miller, Orosz, and Scheinerman \cite{MillerOroszScheinerman2010} developed a fringe construction in which suitable endpoint configurations are joined by a sufficiently dense middle to produce MSTD sets.

Do et al. \cite{DoKulkarniMillerMoonWellens2015} extended these results to pairs of correlated random sets. They discovered the same existence of a positive limiting proportion of sum-dominant sets.

We extend this to a self-correlated process, where the inclusion of $i$ in $A$ depends on the inclusion status of $i-1$.

\begin{definition}\label{MarkovDefn}
    Let $(X_i)_{i\ge0}$ be a stationary two-state Markov chain on $\{0,1\}$ with transition matrix
    \begin{equation}
    P\ =\ \begin{pmatrix}
    p & 1-p \\
    1-q & q
\end{pmatrix}, \quad p,q\in(0,1).
    \end{equation}
\end{definition}
We include the element $i$ in $A$ exactly when $X_i=1$ in the above chain.
\begin{definition}
    Define $A_n \defeq \{i \in [0,n]: X_i = 1\}$, where $X_i$ is generated by the Markov chain in Definition \ref{MarkovDefn}.
\end{definition}

\begin{remark}
    Note that for $p = q = {1}/{2}$ the problem reduces to a set generated by independent fair coin tosses. Furthermore, for $p = 1 - q$, the problem is equivalent to generating the set with a biased coin with inclusion probability $q$.
\end{remark}

The Markov process with transition matrix $P$ is aperiodic and irreducible because $p,q\in(0,1)$, and the finite state space gives positive recurrence. Hence, it admits a stationary distribution $\pi$, which is
\begin{equation}\pi\ =\ \left(\frac{1-q}{2-p-q},\frac{1-p}{2-p-q}\right).\end{equation}
We take $X_0\sim\pi$ to start our chain in its stationary state.

It helps to state a version of the convergence-to-stationarity theorem, appearing in texts such as \cite{levin2017markov}. For probability distributions $\mu$ and $\nu$ on a countable state space
$\Omega$, define their total variation distance by
\begin{equation}\label{TotVarDefn}
 \|\mu-\nu\|_{\mathrm{TV}}
 \ \defeq\ \frac12\sum_{x\in\Omega}|\mu(x)-\nu(x)|.
\end{equation}
\begin{theorem}[Convergence to stationarity]\label{ConvergenceTheorem}
Let $(X_n)_{n\geq 0}$ be an irreducible, aperiodic, positive recurrent Markov chain with transition matrix $P$ and stationary distribution $\pi$. Then, for every initial state $x$,
\begin{equation}
\|P^n(x,\cdot)-\pi\|_{\mathrm{TV}}\ \to\  0
\qquad \text{as } n \to \infty.
\end{equation}
If the state space is finite, then the convergence is geometric: there exist constants $C<\infty$ and $\rho\in(0,1)$ such that
\begin{equation}
\|P^n(x,\cdot)-\pi\|_{\mathrm{TV}}
\ \leq\  C\rho^n
\qquad \text{for all } n \geq  0.
\end{equation}

\end{theorem}
Our main result mirrors the positive-probability results of Martin-O'Bryant \cite{martin2006mstd}.
\begin{theorem}\label{ManyMarkovSetsHaveMoreSums}
    Let $A_n\subseteq S \defeq \{0,\ldots,n\}$ be generated by the Markov chain $(X_i)$. Let $\rho_+^{(n)},\rho_-^{(n)},\text{ and }\rho_0^{(n)}$ be the probabilities that $A_n$ is Sum-dominant, Difference-dominant, and Balanced respectively. As $n\to\infty$, $\rho_+^{(n)},\rho_-^{(n)},\text{ and }\rho_0^{(n)}$ tend to strictly positive limits.
\end{theorem}
\section{Markov-generated sets exhibit MSTD, MDTS, and balanced behavior}
We now prove Theorem \ref{ManyMarkovSetsHaveMoreSums}.
In order to evaluate sum and difference sets for large $n$, we take advantage of the fact that we can analyze their sizes by exclusively looking at left and right fringes.
\begin{definition}[Fringes]
    Let $A_n\subseteq S$ and fix some $K<n/2$. Define the \textit{left} and \textit{right $K$-fringes} respectively as 
    \begin{align}
         &L\ \defeq\  A_n\cap[0,K]&R\ \defeq\  (n-A_n)\cap[0,K]. 
    \end{align}
\end{definition}
Note that elements in $R$ record distances from $n$ of elements in $A_n$
lying in $[n-K,n]$.
\begin{definition}[Fringe score]
    For $A_n\subseteq S$ with left and right fringes as defined above, define its \textit{fringe score} 
    \begin{equation}\Delta(L,R)\ \defeq\ \lvert(L+L)\cap[0,K]\rvert+\lvert(R+R)\cap[0,K]\rvert-2\lvert(L+R)\cap[0,K]\rvert.\end{equation}
\end{definition}
The fringe score captures the difference between the cardinalities of the fringes of the sets $A_n+A_n$ and $A_n-A_n$. Under specified conditions, the fringe score gives the difference in cardinalities between $A_n+A_n$ and $A_n-A_n$. We use a characterization by Miller, Orosz, and Scheinerman \cite{MillerOroszScheinerman2010} and adapt Lemma 2.5 from Zhao \cite{Zhao2011}.
\begin{lemma}\label{FringeFormula}
    Let the middles of $A_n+A_n$ and $A_n-A_n$ be full, i.e. $[K+1,2n-K-1]\subseteq A_n+A_n$ and $[K+1-n,n-K-1]\subseteq A_n-A_n$, then
    \begin{equation}\Delta\ =\ \lvert A_n+A_n\rvert-\lvert A_n-A_n\rvert.\end{equation}
\end{lemma}
\begin{proof}
    Split the sumset $A_n+A_n\subseteq[0,2n]$ into components
    \begin{align}
        [0,K]&&&[K+1,2n-K-1]&[2n-K,2n].
    \end{align}
    Observe that low sums can be written in terms of the left fringes 
    \begin{equation}\lvert(A_n+A_n)\cap[0,K]\rvert\ =\ \lvert(L+L)\cap[0,K]\rvert.\end{equation}
    Similarly, we may determine the high sums by the right fringes, using
    \begin{equation}\lvert(A_n+A_n)\cap[2n-K,2n]\rvert\ =\ \lvert(R+R)\cap[0,K]\rvert.\end{equation}
    Hence, using our full middle, we have that 
    \begin{equation}\lvert A_n+A_n\rvert\ =\ \lvert(L+L)\cap[0,K]\rvert+\lvert(R+R)\cap[0,K]\rvert+2n-2K-1.\end{equation}
    Now consider the difference set $A_n-A_n\subseteq[-n,n]$, and split it into three components
    \begin{align}
        [-n,K-n]&&&[K+1-n,n-K-1]&[n-K,n].
    \end{align}
    Since the right fringe is reflected, for $l\in L$ and $r\in R$ we have
\begin{equation}\label{ReflectedFringeDifference}
 (n-r)-l \ =\ n-(l+r).
\end{equation}
The positive extreme differences are determined by $L+R$, and by symmetry, the negative extreme differences are also determined by $L+R$. Hence, once again using the full middle, we have
    \begin{equation}\lvert A_n-A_n\rvert\ =\ 2\lvert(L+R)\cap[0,K]\rvert+2n-2K-1.\end{equation}
    Subtracting the two identities proves the lemma.
\end{proof}
Lemma \ref{FringeFormula} allows us to compute the sum-difference gap given that we have a full middle. It remains to show that there is a high probability that the middle is full.
\begin{lemma}\label{FullMiddleIsLikely}
    Let $F_{K,n}$ be the event that $[K+1,2n-K-1]\subseteq A_n+A_n$ and $[K+1-n,n-K-1]\subseteq A_n-A_n$.

    There exist constants $C>0$, $0<\theta<1$, depending only on $p$ and $q$, such that for every $K<n/2$,
    \begin{equation}\bbP(F_{K,n}^c)\ \leq\  C\theta^K.\end{equation}
\end{lemma}
\begin{proof}
    By $0<p,q<1$, every transition of the Markov process has strictly positive probability. Hence, there is a constant $\delta>0$ independent of $n$ such that the following statement holds:
    \begin{center}
        Pick any $i,j$ (not necessarily distinct), then conditioned on states $\{X_k=x_k\mid k\neq i,j\}$ for any $x_k$'s, \\$X_i=X_j=1$ with probability $\geq \delta$.
    \end{center}
    For pairwise disjoint representations, let $H_j$ be the event
that both entries of the $j$\textsuperscript{th} pair are $1$. Previous failures
depend only on states outside this pair. The above lower bound then gives
\begin{equation}
\mathbb P\left(\bigcap_{j=1}^M H_j^c\right)
\ \le\  (1-\delta)^M.
\end{equation}
The same argument applies to singleton representations.

    Fix a possible sum $s$. If $s=i+j$ and $X_i=X_j=1$, then $s\in A_n+A_n$. Any $s\in[K+1,2n-K-1]$ has at least $c\min(s+1,2n-s+1)$ pairwise disjoint representations of the form
    \begin{equation}s\ =\ i+j\end{equation}
    where $c>0$ is an absolute constant. We obtain this by extracting distinct unordered pairs from $s=0+s=1+(s-1)=\cdots=s+0$ when $s\leq n$, and from $s=(s-n)+n=\cdots=n+(s-n)$ when $s>n$. Since each unordered pair has at most $2$ orderings, we may choose, say, $c=1/2$.
    Therefore,
    \begin{equation}\bbP(s\notin A_n+A_n)\ \leq\  (1-\delta)^{c\min(s+1,2n-s+1)}\ =\ \theta_1^{\min(s+1,2n-s+1)}\end{equation}
    where $\theta_1\defeq (1-\delta)^c\in(0,1)$ is independent of $n$.
    Summing over all middle sums and using the standard union bound gives
    \begin{align}\bbP\big([K+1,2n-K-1]\not\subseteq A_n+A_n \big)\ &\leq\ 
    \sum_{s=K+1}^{2n-K-1}\theta_1^{\min(s+1,2n-s+1)}\notag\\
    &\leq\  2\sum_{s\geq K+1}\theta_1^s=C_1\theta_1^K
    \end{align}
    where $C_1=\frac{2\theta_1}{1-\theta_1}>0$ is also independent of $n$.

    We apply the same argument to differences. Take a fixed difference $d=i-j$. We have $d\in A_n-A_n$ if $X_i=X_j=1$.

    Consider a middle difference $d\in[K+1-n,n-K-1]$. There are at least $c(n-|d|)$ pairwise disjoint representations for $d$, counted in the same way as for sums, so
    \begin{equation}\bbP(d\notin A_n-A_n)\ \leq\  (1-\delta)^{c(n-|d|)}.\end{equation}
    Summing over all possible $d$'s in the middle gives
     \begin{equation}\bbP\big([K+1-n,n-K-1]\not\subseteq A_n-A_n \big)\ \leq\  C_2\theta_2^K\end{equation}
    for some $C_2>0$ and $0<\theta_2<1$, both independent of $n$, for the same reason as for the sumset case.

    Combining both bounds gives
    \begin{equation}\bbP(F_{K,n}^c)\ \leq\  C\theta^K\end{equation}
    for suitable $C>0,0<\theta<1$, both independent of $n$.
\end{proof}
This lemma provides the tool needed to prove that limits of $\rho_+^{(n)},\rho_-^{(n)},\text{ and }\rho_0^{(n)}$ exist as $n\to\infty$.
\begin{proposition}\label{ProportionLimitsExist}
    The following three limits exist:
    \begin{align}
        \rho_+\  &=\ \lim_{n\to\infty}\rho_+^{(n)}\\
        \rho_-\  &=\ \lim_{n\to\infty}\rho_-^{(n)}\\
        \rho_0\  &=\ \lim_{n\to\infty}\rho_0^{(n)}.
    \end{align}
\end{proposition}
\begin{proof}
    By Lemma \ref{FringeFormula}, on $F_{K,n}$,
    \begin{equation}\lvert A_n+A_n\rvert - \lvert A_n - A_n \rvert\ =\ \Delta(L_K(A_n), R_K(A_n))\end{equation}
    where $L_K(A_n)$ and $R_K(A_n)$ denote the left and right $K$-fringes of $A_n$.

    Let $\star\in\{>,<,=\}$ be a sign. By Lemma \ref{FullMiddleIsLikely},
    \begin{align}
        \bigg| \bbP\big(\lvert A_n+A_n\rvert - \lvert A_n - A_n\rvert\star 0 \big) -\bbP\big(\Delta(L_K(A_n), R_K(A_n))\star 0\big)\ \bigg|&\ \leq\  \bbP(F_{K,n}^c)\notag\\
        &\ \leq C\theta^K.
    \end{align}
    The event $\Delta(L_K(A_n), R_K(A_n))\star 0$ is determined by the endpoint blocks $X_0,\ldots,X_K$ and $X_{n-K},\ldots,X_n$. The finite-state Markov chain is irreducible, aperiodic, and stationary, so the two endpoint blocks become asymptotically independent as $n\to\infty$. Hence \begin{equation}\ell_K\ \defeq\ \lim_{n\to\infty}\bbP(\Delta(L_K(A_n), R_K(A_n))\star 0)\end{equation} exists for each fixed $K$. We may now write 
    \begin{equation}\ell_K-C\theta^K\ \leq\ \liminf_{n\to\infty}\bbP\big(\lvert A_n+A_n\rvert - \lvert A_n - A_n\rvert\star 0 \big)\ \leq\ \limsup_{n\to\infty}\bbP\big(\lvert A_n+A_n\rvert - \lvert A_n - A_n\rvert\star 0 \big)\ \leq\  \ell_K+C\theta^K.\end{equation}
    In other words,
    \begin{equation}\limsup_{n\to\infty}\bbP\big(\lvert A_n+A_n\rvert - \lvert A_n - A_n\rvert\star 0 \big)-\liminf_{n\to\infty}\bbP\big(\lvert A_n+A_n\rvert - \lvert A_n - A_n\rvert\star 0 \big)\ \leq\  2C\theta^K.\end{equation}

    Now, $C\theta^K\to0$ as $K\to\infty$ and is independent of $n$, so for each $\star\in\{>,<,=\}$, 
    \begin{equation}\lim_{n\to\infty}\bbP(\lvert A_n+A_n\rvert - \lvert A_n - A_n\rvert\star 0 )\end{equation}
    exists.
\end{proof}
It remains to prove that the three limits are strictly positive. Our approach involves constructing a positive-probability family of sets with fixed fringes that guarantee one of the three behaviors. For this, we require an extension of Lemma $\ref{FullMiddleIsLikely}$.
\begin{lemma}[Full middle is likely with rich fringes]\label{FullMiddleIsLikelyWithRichFringes}
    Fix $m\geq 0$. Let $K\geq 2m+4$, and let sets $L,R\subseteq[0,K]$ both contain $[m,K]$.

    Let $E_{K,n}(L,R)$ be the event that $L$ and $R$ are the left and right $K$-fringes of $A_n$.

    Then there exist constants $D>0$ and $0<\Theta<1$ dependent only on $p,q,$ and $m$ such that for all $n\geq4K+2$,
    \begin{equation}\bbP(F^c_{K,n}\mid E_{K,n}(L,R))\ \leq\  D\Theta^K.\end{equation}
\end{lemma}

\begin{proof}
    Denote $B=[K+1,n-K-1]$. We observe that $E_{K,n}(L,R)$ fixes the endpoint blocks $[0,K]$ and $[n-K,n]$ of $A_n$.

    Similar to Lemma \ref{FullMiddleIsLikely}, for every $T\subseteq B$ with $|T|\leq 2$, there exists $\eta>0$ dependent only on $p,q$ such that conditioned on all states outside $T$,
    \begin{equation}\bbP(X_i=1\ \forall\ i\in T\mid \text{outside }T)\ \geq\  \eta.\end{equation}
    Suppose there are $M$ representations whose unfixed entries
are pairwise disjoint, with every fixed entry equal to $1$. Applying the preceding conditional bound successively to the unfixed entries gives
    \begin{equation}\bbP(\text{The sum/difference is missing}\mid E_{K,n}(L,R))\ \leq\  (1-\eta)^M.\end{equation}

    First consider the case of sums. Since $[m,K]\subseteq L$, and $K\geq 2m$, we have $[K+1,2K]\subseteq A_n+A_n$. Similarly, as $[m,K]\subseteq R$, we have $[2n-2K,2n-K-1]\subseteq A_n+A_n$. We therefore only need to treat sums in $[2K+1,2n-2K-1]$ probabilistically.

    Suppose $s\in[2K+1,2n-2K-1]$. Consider $h(s)=\min\{s-(2K+1),2n-2K-1-s\}$. We claim that $s$ has at least $c\max(K,h(s)+1)$ pairwise disjoint representations, where $c>0$ is some absolute constant.

    Suppose $s$ is within distance $K$ of the left endpoint $2K+1$, then for
each $r\in[m,K]$, we have the representation
\begin{equation}\label{SumRepresentation}
 r+(s-r) \ =\ s,
\end{equation}
    noting that $s-r\in B$. This gives at least $cK$ disjoint summand pairs. Mirroring the argument, we have at least $cK$ disjoint summand pairs if $s$ is within distance $K$ of the right endpoint $2n-2K-1$.

    Suppose that $s$ is distance $\geq K$ from each endpoint. Then we may choose $\geq ch(s)$ disjoint summand pairs of the form
    \begin{equation}i+j\ =\ s,\quad i,j\in B.\end{equation}
    This completes our claim. It follows that there exists some $0<\lambda<1$ such that for every $s\in[2K+1,2n-2K-1]$,
    \begin{equation}\bbP(s\notin A_n+A_n\mid E_{K,n}(L,R))\ \leq\  \lambda^{\max(K,h(s)+1)}.\end{equation}

    We sum over all possible values of $s$. For each value $h=h(s)$, there are at most two corresponding values of $s$, so we may write
    \begin{align}\bbP\big([K+1, 2n-K-1]\not\subseteq A_n+A_n\mid E_{K,n}(L,R)\big)\ &\leq\  2\sum_{h\geq 0}\lambda^{\max(K,h+1)}\notag\\
    & \leq\  D_1\Theta_1^K
    \end{align}
    where $D_1>0$ and $0<\Theta_1<1$ are constants depending only on $p,q,m$ (noting that we have been using $K\geq 2m$).

    For differences, by symmetry it suffices to show that we can fill the positive middle differences $[0,n-K-1]$. We note that $[m,K]$ is contained in both $L$ and $R$, so $[m,K]\cup \bigl(n-[m,K]\bigr)\subseteq A_n$. Hence
    \begin{equation}[n-2K,n-2m]\ \subseteq\  A_n-A_n.\end{equation}
    Using $K\geq 2m$, we obtain
    \begin{equation}[n-2K,n-K-1]\ \subseteq\  A_n-A_n.\end{equation}

    Thus we only need to probabilistically consider differences in $[0,n-2K-1]$. Define $g(d)=n-2K-1-d$, the distance of $d$ from the upper endpoint of the given interval. 

    If $g(d)<K$, then $d$ is close to the upper endpoint. For each $r\in[m,K]$, we have the representation
\begin{equation}\label{DiffRepresentation}
 d \ =\ (n-r)-(n-r-d),
\end{equation}

    noting that $n-r\in A_n$ and, for $j\defeq n-r-d$, $j\in B$. Hence, the possible values of $j$ give $\geq cK$ disjoint difference pairs.

    If $g(d)\geq K$, then $d$ is far away from the upper endpoint, so there are $\geq cg(d)$ disjoint difference pairs of the form
    \begin{equation}i-j\ =\ d,\quad i,j\in B.\end{equation}
    Thus there exists some $0<\lambda<1$ such that 
    \begin{equation}\bbP(d\notin A_n-A_n\mid E_{K,n}(L,R))\ \leq\  \lambda^{\max(K,g(d))}.\end{equation}
    Summing over $d$ gives the inequality
     \begin{equation}\bbP([0,n-K-1]\not\subseteq A_n-A_n\mid E_{K,n}(L,R))\ \leq\, \ \sum_{g\geq 0}\lambda^{\max(K,g)}\ \leq\, \ D_2\Theta_2^K,\end{equation}
    where $D_2>0$ and $0<\Theta_2<1$ are once again dependent only on $p,q,m$. Since $A_n-A_n$ is symmetric, the same bound works for negative differences.

    Combining, we get that for some constants $D>0$ and $0<\Theta<1$ depending only on $p,q,m$, 
    \begin{equation}\bbP(F_{K,n}^c|E_{K,n}(L,R))\ \leq\ D\Theta^K.\end{equation}
\end{proof}
We now have the tools to prove Theorem \ref{ManyMarkovSetsHaveMoreSums}. For each case, conditional on a rich fringe, the middle is likely to be full with a large enough $K$.
\begin{proof}[Proof of Theorem \ref{ManyMarkovSetsHaveMoreSums}]
    We break the theorem down into three cases, covering sum-dominant, difference-dominant, and balanced proportions respectively.

    Pick $K\geq 12$, and define the fringes
    \begin{align}
        &L^+_K\ \defeq\ \{0,2,3,7,8,9,10\}\cup[12,K] &R^+_K\ \defeq\ \{1,2,3,6,8,9,10,11\}\cup[12,K].
    \end{align}
    A direct calculation gives $(L^+_K+L^+_K)\cap[0,K]=[0,K]\setminus\{1\}$,$(R^+_K+R^+_K)\cap[0,K]=[0,K]\setminus\{0,1\}$, and $(L^+_K+R^+_K)\cap[0,K]=[0,K]\setminus\{0,7\}$.
    Hence 
    \begin{equation}\Delta(L_K^+,R_K^+)\ =\ 1.\end{equation}
    Define $E_{K,n}^+$ as the event that $A_n$ has left and right $K$-fringes $L_K^+$ and $R_K^+$. This event specifies two finite binary strings, and by asymptotic independence of the two endpoint blocks, we have
    \begin{equation}\lim_{n\to\infty}\bbP(E^+_{K,n})\ >\ 0.\end{equation}
   Lemma 2.6 applies with $m=12$. Choose a fixed $K\geq28$ large enough that $D\Theta^K\leq1/2$.
    \begin{equation}\bbP(F_{K,n}|E_{K,n}^+)\ \geq\  1/2\text{ for all }n\geq4K+2.\end{equation}
    We conclude that 
    \begin{equation}\liminf_{n\to\infty}\bbP(F_{K,n}\cap E_{K,n}^+)\ >\ 0.\end{equation}
    On the event, $F_{K,n}\cap E_{K,n}^+$, sums dominate, so
    \begin{equation}\rho_+\ >\ 0.\end{equation}

    We make the same argument for difference-dominant and balanced cases. For difference-dominance, we pick $K\geq 5$ and consider the fringes
    \begin{align}
        &L_K^-\ \defeq\ \{0,2\}\cup[5,K]&R_K^-\ \defeq\ \{0,1\}\cup [5,K].
    \end{align}
    This gives $(L_K^-+L_K^-)\cap[0,K]=[0,K]\setminus\{1,3\}$, $(R_K^-+R_K^-)\cap[0,K]=[0,K]\setminus\{3,4\}$, and $(L_K^-+R_K^-)\cap[0,K]=[0,K]\setminus\{4\}$, so
    \begin{equation}\Delta(L_K^-,R_K^-)\ =\ -2.\end{equation}
    Let $E_{K,n}^-$ be the corresponding fringe event. Using an identical argument as above, except with $m=5$ instead, for a sufficiently large fixed $K$ we have
    \begin{equation}\liminf_{n\to\infty}\bbP(F_{K,n}\cap E_{K,n}^-)\ >\ 0\end{equation}
    and since differences dominate in this event,
    \begin{equation}\rho_-\ >\ 0.\end{equation}

    Finally, we consider balanced fringes
    \begin{equation}L_K^0\ \defeq\ [0,K], \qquad R_K^0\ \defeq\ [0,K].\end{equation}
    We obtain $\Delta(L_K^0,R_K^0)=0$. Once again denoting corresponding fringe event $E_{K,n}^0$, we use the same argument with $m=0$ to show that for a sufficiently large fixed $K$,
    \begin{equation}\liminf_{n\to\infty}\bbP(F_{K,n}\cap E_{K,n}^0)\ >\ 0\end{equation}
    and since we have the balanced case in this event,
    \begin{equation}\rho_0\ >\ 0.\end{equation}
\end{proof}
\section{Differences dominate on average}

\begin{theorem}\label{Expected difference}
    Let $A_n\subseteq S \defeq \{0,\ldots,n\}$ be generated by a Markov process $\{X_i\}$. As $n \to \infty$,
    \begin{equation}
        \E\big[\lvert A_n-A_n\rvert - \lvert A_n+A_n\rvert\big] \to \Lambda(p,q)\ \defeq\  \frac{2(1-q)(1+p-p^2-pq)}{(1-p)^2(p+q)(2-p-q)}.
    \end{equation}

\end{theorem}

To prove this theorem, we first reduce the problem to fringe probabilities.
Let $(Y_i)_{i\geq0}$ and $(Z_i)_{i\geq0}$ be independent stationary
copies of $(X_i)_{i\geq0}$. For $k\geq0$, define
\begin{equation}\label{eq:w-definition}
    w_k \ \defeq\
    \mathbb{P}\bigl(Y_iZ_i=0 \text{ for every } 0\leq i\leq k\bigr).
\end{equation}
\begin{lemma}\label{simplyfing expected difference}
    Define $E_n\defeq\mathbb{E}[\lvert A_n-A_n\rvert - \lvert A_n+A_n\rvert\big]$. Then
    \begin{equation}
    \lim_{n \to \infty} E_n \ =\  2\sum_{k=0}^\infty (u_k - w_k).
    \end{equation}
\end{lemma}

\begin{proof}
    First, let's rewrite this problem in terms of missing sums and differences. Notice that
    \begin{equation}
    |A_n-A_n| - |A_n +A_n|\  =\  |(A_n+A_n)^c| - |(A_n-A_n)^c|
    \end{equation}
    where complements of $A_n+A_n$ and $A_n-A_n$ are taken in $[0,2n]$ and $[-n,n]$, respectively. By linearity of expectation, we obtain
    \begin{equation}\label{MissingSumsDiffs}
        E_n \ =\  \E[\textbf{missing sums}] - \E[\textbf{missing differences}].
    \end{equation}
    For $k\ge0$, define
\begin{equation}
u_k\ =\ \mathbb P(k\notin A_k+A_k).
\end{equation}
For every $n\ge k$, this also equals
$\mathbb P(k\notin A_n+A_n)$, since only $X_0,\ldots,X_k$
determine whether $k$ is present. By stationarity and
reversibility,
\begin{equation}
\mathbb P(2n-k\notin A_n+A_n)\ =\ u_k,
\qquad 0\le k\le n.
\end{equation}
    Let's define the probability that $n-k$ is missing in $A_n - A_n$, for $0 \le k \le n-1$, as
    \begin{equation}
    v_{n,k}\ \defeq\ P(n-k \notin A_n - A_n).
    \end{equation}
    The difference $n-k$ is missing exactly when there is no $i \in [0,k]$ such that
    \begin{equation}
    X_i\ =\ X_{n-k+i}\ =\ 1.
    \end{equation}
    Note that the difference $0$ is missing only if $A_n = \emptyset$. Hence
    \begin{equation}
    P(0 \notin A_n - A_n)\ =\ P(A_n = \emptyset)\ =\ \pi_0p^n.
    \end{equation}
    The probability of not including the sum $k$ is $u_k$, and likewise for $2n-k$. The probability of not including $n-k$ in the difference set is $v_{n,k}$, and likewise for $k-n$. Therefore
    \begin{equation}\label{eqn:FiniteExpectIdentity}
    E_n\ =\ 2\sum_{k=0}^{n-1}(u_k - v_{n,k}) + u_n - \pi_0 p^n.
    \end{equation}

For fixed $k$, the blocks $(X_0,\ldots,X_k)$ and
$(X_{n-k},\ldots,X_n)$ become independent as $n\to\infty$.
Hence $v_{n,k}\to w_k$.
We now bound the tails uniformly in $n$. Let $\delta$ be
the conditional lower bound from Lemma \ref{FullMiddleIsLikely}. There are
$\lfloor k/2\rfloor+1$ disjoint representations of the
sum $k$, allowing the singleton pair when $k$ is even.
For $0\le k<n$, put $d=n-k>0$. The graph with edges
$\{i,i+d\}$, $0\le i\le k$, is a union of paths, so it
contains at least $(k+1)/2$ disjoint edges. Therefore
\begin{equation}
u_k\ \le\ (1-\delta)^{(k+1)/2},\qquad
v_{n,k}\ \le\ (1-\delta)^{(k+1)/2}.
\end{equation}
Taking $n\to\infty$ gives the same bound for $w_k$.
Thus, for some $C>0$ and $0<\alpha<1$,
\begin{equation}
u_k+v_{n,k}+w_k\le C\alpha^k,\qquad 0\le k<n.
\end{equation}
For $n>K$, it follows that
\begin{equation}
\left|\sum_{k=K+1}^{n-1}(u_k-v_{n,k})\right|
\ \le\  C\sum_{k>K}\alpha^k.
\end{equation}
Set $S\defeq2\sum_{k=0}^{\infty}(u_k-w_k)$. The preceding geometric bounds
show that this series converges absolutely. By
Equation \eqref{eqn:FiniteExpectIdentity}, for $n>K$,
\begin{equation}\label{eq:double-limit-bound}
\begin{aligned}
 |E_n-S|
 &\ \leq\ 2\sum_{k=0}^{K}|v_{n,k}-w_k|
 +2\sum_{k=K+1}^{n-1}(u_k+v_{n,k})\\
 &\qquad +2\sum_{k=K+1}^{\infty}(u_k+w_k)
 +u_n+\pi_0p^n.
\end{aligned}
\end{equation}
For fixed $K$, the first sum tends to zero as $n\to\infty$, since
$v_{n,k}\to w_k$ for each $k$. Also, $u_n+\pi_0p^n\to0$. The geometric
bounds on the remaining sums give a constant $C'>0$, independent of
$K$, such that
\begin{equation}\label{eq:double-limit-tail}
 \limsup_{n\to\infty}|E_n-S|
 \ \leq\ C'\sum_{k=K+1}^{\infty}\alpha^k
 \ =\ \frac{C'\alpha^{K+1}}{1-\alpha}.
\end{equation}
Letting $K\to\infty$, we obtain
\begin{equation}\label{eq:expectation-series-limit}
 \lim_{n\to\infty}E_n \ =\ 2\sum_{k=0}^{\infty}(u_k-w_k).
\end{equation}

\end{proof}

\begin{lemma}\label{counting manipulation of expected}
Define the matrix
\begin{equation}
Q = \begin{pmatrix}
    p^2 &p(1-p) & p(1-p)\\
    p(1-q) & pq & (1-p)(1-q)\\
    p(1-q) & (1-p)(1-q) & pq
\end{pmatrix}
\end{equation}
and vectors
\begin{equation}
e_{00} = \begin{pmatrix}
    1\\
    0\\
    0
\end{pmatrix}, \quad 
\textbf{1} = \begin{pmatrix}
    1\\
    1\\
    1
\end{pmatrix}, \quad \nu = (\pi_0^2, \pi_0\pi_1, \pi_0\pi_1), \quad \eta = (\pi_0 p, \pi_0(1-p), \pi_1 (1-q)).
\end{equation}
Then
\begin{equation}
u_{2m} = \pi_0e_{00}^T Q^m \textbf{1}, \quad u_{2m+1} = \eta Q^m \textbf{1},  \quad w_k = \nu Q^k \textbf{1}.
\end{equation}

\end{lemma}

\begin{proof}

Let the allowed pair-state space be
\begin{equation}
\epsilon = \{00,01,10\}
\end{equation}
ordered as such. Then
\begin{equation}
Q = \begin{pmatrix}
    p^2 &p(1-p) & p(1-p)\\
    p(1-q) & pq & (1-p)(1-q)\\
    p(1-q) & (1-p)(1-q) & pq
\end{pmatrix}
\end{equation}
represents the substochastic matrix between those pair-state spaces.

The reverse transition matrix of a stationary Markov chain is defined by
\begin{equation}
P^*_{a,c}\ =\  \bbP(X_{i-1 }=c | X_{i}=a)\  =\ \frac{\pi_c P_{c,a}}{\pi_a}.
\end{equation}
For the present two-state chain,
\begin{equation}
\pi_0P_{0,1}\ =\ \frac{(1-q)(1-p)}{2-p-q}\ =\ \pi_1P_{1,0}.
\end{equation}
Thus we have the detailed-balance equations,
\begin{equation}
\pi_aP_{a,c}\ =\ \pi_cP_{c,a},
\end{equation}
which hold for every $a,c \in \{0,1\}$, and hence
\begin{equation}
P^*\ =\ P.
\end{equation}
Therefore, the chain is reversible. Now, suppose we have a finite interval from $X_l$ to $X_r$ with $X_l = a$, $X_r = b$. By the Markov property, and by expanding to the variables $X_{l-1}$ and $X_{r+1}$, it follows that
\begin{equation}
\bbP(X_{l-1} = c, X_{r+1} = d\mid X_l, X_r)\ =\ P_{a,c}P_{b,d}.
\end{equation}
Note that we're restricting to the states $00,01,10$ since, for the event $2m \notin A+A$, we need the pairs $(m,m),(m-1,m+1),\dots ,(0,2m)$ not to appear. Precisely, $(m-i,m+i)$ cannot have state $11$ for all $0\le i \le m$. Since we start counting at $(m,m)$ with state $00$, we obtain
\begin{equation}
u_{2m}\ =\ \pi_0e_{00}^T Q^m \textbf{1}.
\end{equation}

For the odd case $k = 2m+1$ to be missing, we need for all $0\le i \le m$ the pair $(m-i,m+1+i)$ to be in one of the states $00,01,10$, starting specifically from $(m,m+1)$. Hence
\begin{equation}
u_{2m+1}\ =\ \eta Q^m \textbf{1}.
\end{equation}
By asymptotic independence of the endpoint blocks,
\begin{equation}
w_k\ =\ \lim_{n\to \infty} \bbP(n - k \notin A_n - A_n).
\end{equation}
Since the blocks $Y_i$ and $Z_i$ are independent, it follows that we can use the transition matrix $Q$ with starting states $00,01,10$. Hence,
\begin{equation}
w_k\ =\ \nu Q^k \textbf{1}.
\end{equation}

\end{proof}

\begin{proof}[Proof of Theorem \ref{Expected difference}]

By Lemmas \ref{simplyfing expected difference} and \ref{counting manipulation of expected}, it follows that, as $n \to \infty$,
\begin{equation}
\lim_{n\to\infty}E_n\ =\ 2\left(\sum_{m=0}^{\infty} \pi_0e_{00}^T Q^m\textbf{1} + \sum_{m=0}^{\infty}\eta Q^m \textbf{1} - \sum_{k=0}^\infty\nu Q^k \textbf{1}\right).
\end{equation}
Since the sum of each row is either $1 - (1-p)^2 < 1$ or $1-q(1-p)$, we know that $||Q||_{\infty} < 1$; therefore, we can simplify and find
\begin{equation}
\sum_{m=0}^{\infty} Q^m\mathbf1\ =\ (I-Q)^{-1}\mathbf1\ =\ \frac{1}{(1-p)^2(p+q)}\begin{pmatrix}
    3p + q-2p^2 - 2pq\\
    p + 1-p^2 - pq\\
    p + 1-p^2 - pq 
\end{pmatrix}.
\end{equation}
Hence, as $n \to \infty$,
\begin{equation}
\lim_{n\to\infty}E_n\ =\ 2(\pi_0e_{00}^T + \eta - \nu)(I-Q)^{-1}\textbf{1}\ =\ \frac{2(1-q)(1+p-p^2-pq)}{(1-p)^2(p+q)(2-p-q)}.
\end{equation}

\end{proof}

\section{The difference formula converges very fast}
\begin{theorem}\label{ConvergeFast}
    There exists a constant $0<\lambda<1$ dependent only on $p$ and $q$ such that
    \begin{equation}E_n\ =\ \Lambda(p,q)+O(n\lambda^n).\end{equation}
\end{theorem}
\begin{proof}
    We recall that 
    \begin{align}
        E_n\ &=\ 2\sum_{k=0}^{n-1}(u_k-v_{n,k})+u_n-\pi_0p^n \\
        \Lambda(p,q)\ &=\ 2\sum_{k=0}^\infty (u_k-w_k).
    \end{align}
    By the triangle inequality,
    \begin{equation}|E_n-\Lambda|\ \leq\  u_n+\pi_0p^n+2\sum_{k\geq n}(u_k+w_k)+2\sum_{k=0}^{n-1}|v_{n,k}-w_k|.\end{equation}
    By Lemma \ref{simplyfing expected difference}, we have $u_k+w_k=O(\alpha^k)$ for some $0<\alpha<1$.

    Without loss of generality, choose $\alpha\geq p$, which gives
    \begin{equation}u_n+\pi_0p^n+2\sum_{k\geq n}(u_k+w_k)\ =\ O (\alpha^n).\end{equation}
    It remains to bound $\sum_{k=0}^{n-1}|v_{n,k}-w_k|$. By Theorem \ref{ConvergenceTheorem}, we may find some fixed $0<\beta<1$ such that
    \begin{equation}|(P^j)_{ab}-\pi_b|=O(\beta^j)\end{equation}
    for any states $a,b\in\{0,1\}$.

    First suppose $0\le k\le n/3$. Conditional on
$X_0,\ldots,X_k$, the distribution of $X_{n-k}$ is
$P^{n-2k}(X_k,\cdot)$. Its total variation distance from
$\pi$ is $O(\beta^{n-2k})$. Extending the second block
by the Markov property does not increase this distance.
Since $n-2k\ge n/3$, we obtain, uniformly in $k$,
\begin{equation}
|v_{n,k}-w_k|\ =\ O(\beta^{n/3}).
\end{equation}
Hence,
\begin{equation}
\sum_{0\le k\le n/3}|v_{n,k}-w_k|
\ =\ O(n\beta^{n/3}).
\end{equation}
For $n/3<k\le n-1$, the bounds in Lemma 3.2 give
\begin{equation}
|v_{n,k}-w_k|\le v_{n,k}+w_k\ =\ O(\alpha^k),
\end{equation}
so
\begin{equation}
\sum_{\substack{0\le k\le n-1\\k>n/3}}
|v_{n,k}-w_k|\ = \ O(n\alpha^{n/3}).
\end{equation}
Taking $\lambda=\max\{\alpha^{1/3},\beta^{1/3}\}<1$
and combining with the preceding $O(\alpha^n)$ bound
completes the estimate
    \begin{equation}u_n+\pi_0p^n+2\sum_{k\geq n}(u_k+w_k)+2\sum_{k=0}^{n-1}|v_{n,k}-w_k|\ =\ O(n\lambda^n).\end{equation}
    This completes the proof.
\end{proof}
\section{Numerical simulations and future work}
We did not obtain a closed form expression for the limiting probabilities in Proposition \ref{ProportionLimitsExist}. Numerical simulations in Figure \ref{fig:Parameters} give some brief insights.
\begin{figure}[H]
\centering
\includegraphics[width=\linewidth]{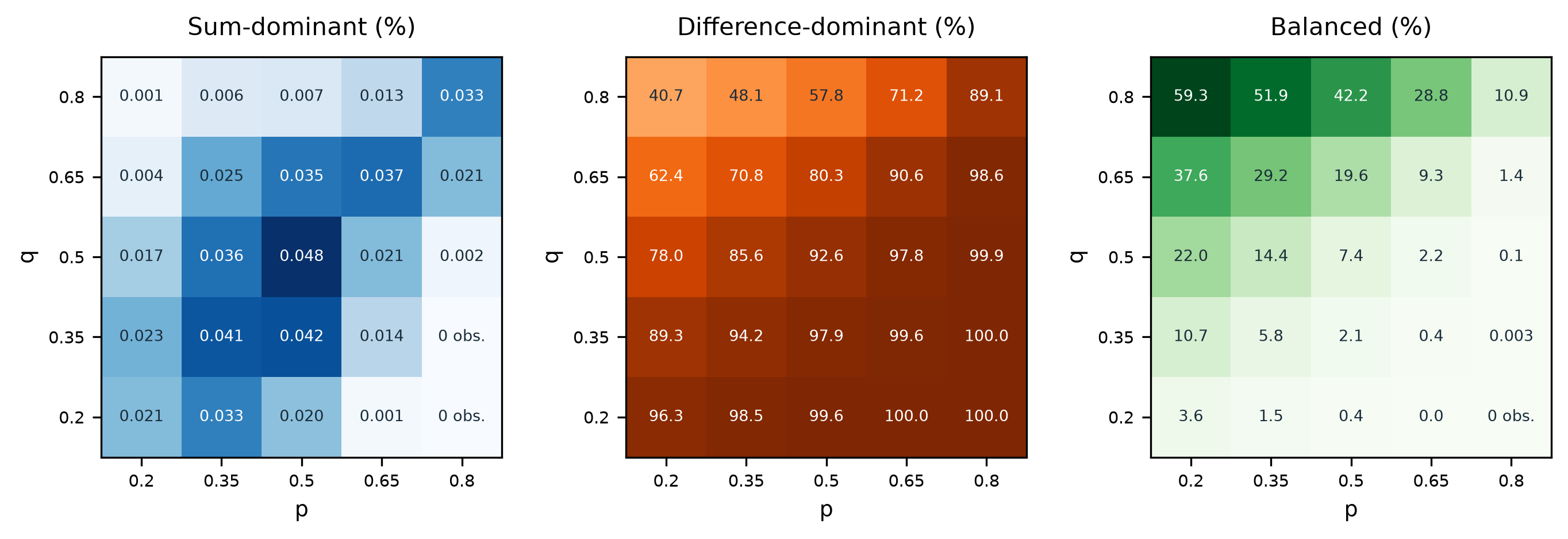}
\caption{Simulations of different $p,q$ with $n=1000$, over 100000 trials per parameter pair.}
\label{fig:Parameters}
\end{figure}

We believe that the uniformly random subset model, corresponding to $p=q=1/2$, maximizes the limiting probability of a sum-dominant set within this family. Further computations and analysis may resolve the following conjecture.
\begin{conj}
    The limiting proportion $\rho_+$ is maximized when $p=q=1/2$.
\end{conj}

We also estimate the probabilities of the three cases for several
values of $n$ and fixed parameter pairs in Figure \ref{fig:SetSize}.
\begin{figure}[H]
\centering
\includegraphics[width=\linewidth]{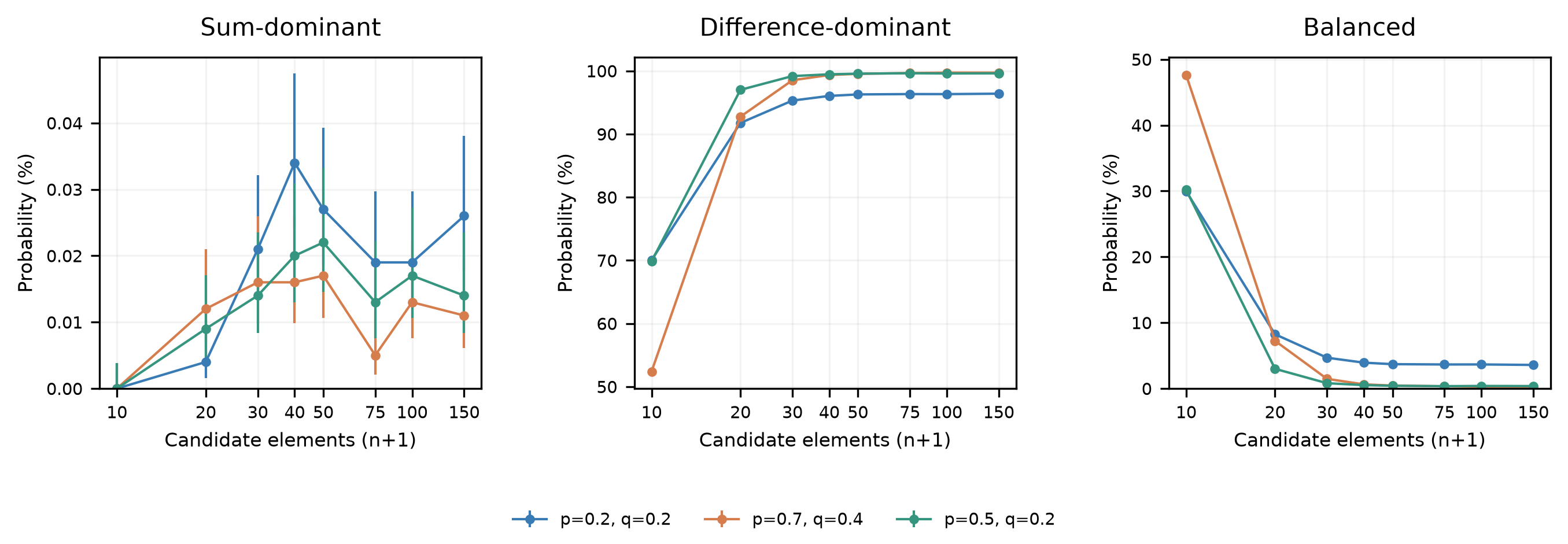}
\caption{Simulations of different $n$ with fixed $p,q$, over 100000 trials per parameter pair. Error bars show
pointwise $95\%$ Wilson confidence intervals.}
\label{fig:SetSize}
\end{figure}
These simulations illustrate the finite-size behavior of the probability pairs. A possible direction for future work is to determine
how the optimal convergence rate in Theorem \ref{ConvergeFast} depends
on $p$ and $q$.
\newpage
\section*{Acknowledgments}
This work was supported by the University of Cambridge, Williams College, the Winston Churchill Foundation,
the Finnerty Fund, and NSF Grant DMS2241623.
\printbibliography

\end{document}